\documentclass{article}

\usepackage{amsmath,amsthm,amssymb,amsfonts}
\usepackage{latexsym,enumerate}
\usepackage{mathrsfs}
\usepackage{epsfig,color}
\usepackage{hyperref}

\def\R{\mathbb{R}}

\def\eps{\varepsilon}
\def\inn{\mathrm{in}}
\def\out{\mathrm{out}}

\def\dd{\mathrm{d}}

\newtheorem{theorem}{Theorem}[section]
\newtheorem{lemma}[theorem]{Lemma}

\theoremstyle{definition}

\theoremstyle{remark}
\newtheorem{remark}[theorem]{Remark}

\title{Finite speed of propagation for a non-local porous medium equation}

 \author{Cyril Imbert\footnote{CNRS, UMR 8050, 
Universit\'e   Paris-Est Cr\'eteil, 
61 av. du G\'en\'eral de Gaulle, 94010 Cr\'eteil, cedex,  France}}

\begin{document}

\maketitle

\begin{abstract}
This note is concerned with proving the finite speed of propagation
for some non-local porous medium equation by adapting  arguments
developed by Caffarelli and V\'azquez (2010). 
\end{abstract}



\section{Introduction}

Caffarelli and V\'azquez \cite{cv10a} proved finite speed of
propagation for non-negative weak solutions of
\begin{equation}\label{eq:cv}
\partial_t u = \nabla \cdot ( u \nabla^{\alpha-1} u), \quad t>0, x \in
\R^d 
\end{equation}
with $\alpha \in (0,2)$ and $\nabla^{\alpha-1}$ stands for $\nabla
(-\Delta)^{\frac\alpha2 -1}$.  We adapt here their proof 
 in order to treat the more general case
\begin{equation}\label{eq:main}
\partial_t u = \nabla \cdot ( u \nabla^{\alpha-1} u^{m-1}), \quad t>0, x \in
\R^d
\end{equation}
for $m > m_\alpha: = 1 + d^{-1} (1-\alpha)_+ + 2 (1-\alpha^{-1})_+$.
Equation~\eqref{eq:main} is supplemented with the following initial
condition
\begin{equation}\label{eq:ic}
 u(0,x) = u_0 (x), \quad x \in \R^d
\end{equation}
for some $u_0 \in L^1(\R^d)$. 
The result contained in this note gives a positive answer to a
question posed in \cite{stv} where finite of infinite speed of
propagation is studied for another generalization of \eqref{eq:cv}.
We recall that weak solutions of \eqref{eq:main}-\eqref{eq:ic} are
constructed in \cite{bik} for $m > m_\alpha$ (see also \cite{bik0}).

In the following statement (and the remaining of the note), $B_R$
denotes the ball of radius $R>0$ centered at the origin.
\begin{theorem}[Finite speed of propagation]\label{thm:main}
  Let $m > m_\alpha$ and assume that $u_0 \ge 0$ is integrable and
  supported in $B_{R_0}$.  Then a non-negative weak solution $u$ of
  \eqref{eq:main}-\eqref{eq:ic} is supported in $B_{R(t)}$ where
\[ R(t) = R_0  + C  t^{\frac{1}{\alpha}} \qquad \text{ with } \qquad
 C = C_0 \|u_0\|_\infty^{\frac{m-1}\alpha} \] for some constant $C_0>0$ only
depending on dimension, $\alpha$ and $m$.
\end{theorem}
\begin{remark}
  The technical assumption $m > m_\alpha$ is imposed to ensure the
  existence of weak solutions; see \cite{bik}. 
\end{remark}
\begin{remark}
In view of the Barenblatt solutions constructed in \cite{bik}, the
previous estimate of the speed of propagation is
optimal. 
\end{remark}

The remaining of the note is organized as follows.  In preliminary
Section~\ref{sec:prelim}, the equation is written in non-divergence
form, non-local operators appearing in it are written as singular
integrals, invariant scalings are exhibited and an approximation
procedure is recalled.  Section~\ref{sec:contact} is devoted to the
contact analysis. A first lemma for a general barrier is derived in
Subsection~\ref{sub:contact}. The barrier to be used in the proof of
the theorem is constructed in Subsection~\ref{sub:error}. The main
error estimate is obtained in
Subsection~\ref{sub:estim}. Theorem~\ref{thm:main} is finally proved
in Section~\ref{sec:proof}.

\paragraph{Notation.} For $a \in \R$, $a_+$ denotes $\max(0,a)$. An
inequality written as $A \lesssim B$ means that there exists a
constant $C$ only depending on dimension, $\alpha$ and $m$ such that
$A \le C B$. If $\alpha \in (0,1)$, a function $u$ is in
$\mathcal{C}^\alpha$ means that it is $\alpha$-H\"older continuous. If
$\alpha \in (1,2)$, it means that $\nabla u$ is $(\alpha-1)$-H\"older
continuous. For $\alpha \in (0,2)$, a function $u$ is in
$\mathcal{C}^{\alpha+0}$ if it is in $\mathcal{C}^{\alpha+\eps}$ for
some $\eps >0$ and $\alpha + \eps \neq 1$.

\section{Preliminaries}
\label{sec:prelim}

The contact analysis relies on writing Eq.~\eqref{eq:main} into the following
non-diver\-gence form
\begin{equation}\label{eq:non-div}
 \partial_t u = \nabla u \cdot \nabla p + u \Delta p 
\end{equation}
where $p$ stands for the pressure term and is defined as 
\[ p = (-\Delta)^{\frac{\alpha}2-1} u^{m-1}.\]
It is also convenient to write $v = u^{m-1} = G(u)$.

We recall that for a smooth and bounded function $v$, the non-local
operators appearing in \eqref{eq:non-div} have the following singular
integral representations,
\begin{align*}
  \nabla (-\Delta)^{\frac\alpha2 -1} v &= c_{\alpha} \int ( v(x+z)-v(x))z \frac{dz}{|z|^{d+\alpha}}, \\
  - (-\Delta)^{\frac\alpha2} v &= \bar{c}_{\alpha} \int (v(x+z)+
  v(x-z) -2 v(x)) \frac{dz}{|z|^{d+\alpha}}.
\end{align*}

The following elementary lemma makes  the scaling of the
equation precise. 
\begin{lemma}[Scaling]\label{lem:scaling}
If $u$ satisfies \eqref{eq:main} then 
$U(t,x) = A u(Tt,Bx)$ satisfies \eqref{eq:main} as soon as 
\[ T = A^{m-1} B^\alpha.\]
\end{lemma}

Consider non-negative solutions of the viscous approximation of
\eqref{eq:main}, i.e.
\begin{equation}\label{eq:main-reg}
\partial_t u = \nabla \cdot ( u \nabla^{\alpha-1} G(u)) + \delta
\Delta u, \quad t>0, x \in \R^d.
\end{equation}
For sufficiently smooth initial data $u_0$, solutions
are at least $\mathcal{C}^2$ with respect to $x$ and $\mathcal{C}^1$ with respect to
$t$.

\section{Contact analysis}
\label{sec:contact}

\subsection{The contact analysis lemma}
\label{sub:contact}

In the following lemma, we analyse what happens when a sufficiently
regular barrier $U$ touches a solution $u$ of \eqref{eq:main-reg} from
above. The monotone term such as $\partial_t u$, $\Delta u$ or
$-(-\Delta)^{\frac\alpha2} u$ are naturally ordered. But this is not
the case for the non-local drift term $\nabla u \cdot \nabla p$. The
idea is to split it is a ``good'' part (i.e. with the same monotony as
$\Delta u$ for instance) and a bad part. It turns out that the bad
part can be controlled by a fraction of the good part; see
\eqref{estim:2} in the proof of the lemma.
\begin{lemma}[Contact analysis]\label{lem:contact}
Let $u$ be a solution of the approximate equation~\ref{eq:main-reg} 
 and $U(t,x)$ be $\mathcal{C}^2((0,+\infty) \times (\R^d \setminus
  B_1))$, radially symmetric w.r.t. $x$, non-increasing
  w.r.t. $|x|$. If
\[ \begin{cases} u \le U \text{ for } (t,x) \text{ in } [0,t_c] \times \R^d, \\
u(t_c,x_c) = U(t_c,x_c),
\end{cases}\]
then 
\begin{equation}\label{eq:contact}
 \partial_t U \le \nabla U \cdot \nabla P  + U \Delta P + \delta
\Delta U + e
\end{equation}
holds at $(t_c,x_c) \in (0,+\infty) \times (\R^d \setminus
  B_1))$
where 
\[ \begin{cases} V = G(U) \\
P  = (-\Delta)^{\frac{\alpha}2-1} V \medskip\\
 e = |\nabla U | (I_{\out,+}(V) - I_{\out,+} (v)) \ge 0 \end{cases}\]
with 
\[  I_{\out,+} (w) = \begin{cases} \int_{\stackrel{|y| \ge \gamma}{y
        \cdot \hat x_c \ge 0}} (w(x_c + y) - w (x_c)) (y \cdot \hat
    x_c) \frac{dy}{|y|^{d+\alpha}} & \text{ if }
    \alpha \ge 1 \medskip\\
    \int_{\stackrel{|y| \ge \gamma}{y \cdot \hat x_c \ge 0}} w(x_c +
    y) (y \cdot \hat x_c) \frac{dy}{|y|^{d+\alpha}} & \text{ if }
    \alpha \in (0,1) \end{cases} \]
(where $\hat x_C = x_C / |x_C|$) 
for $\gamma$ such that 
\[ c_\alpha \gamma |\nabla U (x_c)| \le \bar{c}_\alpha U (x_c)\]
where $c_\alpha$ and $\bar{c}_\alpha$ are the constants appearing in
the definitions of the two non-local operators. 
\end{lemma}
\begin{proof}
At the contact point $(t_c,x_c)$, the following holds true 
\begin{align*}
\partial_t u &\ge \partial_t U \\
\nabla u &= \nabla U = -|\nabla U| \hat x_c \\
\Delta u &\le \Delta U.
\end{align*}
This implies that 
\begin{equation}\label{eq:rewrite}
 \partial_t U \le \nabla U \cdot \nabla p + U \Delta p + \delta
\Delta U.
\end{equation}
We next turn our attention to $\nabla p$ and $\Delta p$. We drop the
time dependence of functions since it plays no role in the remaining
of the analysis. 

The fact that $U$ is radially symmetric and non-decreasing implies in
particular that $\nabla U (x)= - |\nabla U(x)| x / |x|$ which in turn implies
\begin{equation}\label{eq:drift}
 \nabla U \cdot \nabla p = - |\nabla U| I(v) 
\end{equation}
where 
\[ I (v) = \begin{cases} 
c_\alpha \int (v(x_c + y) - v (x_c)) (y \cdot \hat x_c)
\frac{dy}{|y|^{d+\alpha}} & \text{ if } \alpha \in [1,2), \bigskip\\
c_\alpha \int v(x_c + y)  (y \cdot \hat x_c)
\frac{dy}{|y|^{d+\alpha}} & \text{ if } \alpha \in (0,1).
\end{cases} \] 
We now split $I$ into several pieces by
splitting the domain of integration $\R^d$ into $B_\gamma^{\text{in},\pm} = \{ y
\in B_\gamma : \pm y \cdot \hat x_c \ge 0\}$ and
$B_\gamma^{\text{out},\pm}= \{ y
\notin B_\gamma : \pm y \cdot \hat x_c \ge 0\}$ for some parameter
$\gamma >0$ to be fixed later. We thus can write 
\[ I (v) =  I_{\text{in},+} (v)+I_{\text{in},-} (v)+I_{\text{out},+}
(v)+I_{\text{out},-} (v)\]
where
\[ I_{\text{in/out},\pm} (v) = \begin{cases}
c_\alpha \int_{B_\gamma^{\text{in/out},\pm}} (v(x_c + y) - v (x_c)) (y \cdot \hat x_c)
\frac{dy}{|y|^{d+\alpha}}& \text{ if } \alpha \in [1,2), \bigskip\\
c_\alpha \int_{B_\gamma^{\text{in/out},\pm}} v(x_c + y) (y \cdot \hat x_c)
\frac{dy}{|y|^{d+\alpha}}& \text{ if } \alpha \in (0,1).
\end{cases}  \] 

We can proceed similarly for $\Delta p$. Remark that 
\begin{equation}\label{eq:diffusion}
\Delta p = J (v)
\end{equation}
 where 
 \[ J (v) = \bar{c}_\alpha \int (v(x+y)+ v (x-y) - 2 v(x))
 \frac{dy}{|y|^{d+\alpha}}.\] We can introduce $J_{\text{in/out},\pm}
 (v)$ analogously.  \bigskip

We first remark that, 
\begin{equation}\label{estim:1}
\begin{cases}
 -I_{\text{in/out},-} (v) \le - I_{\text{in/out},-} (V) \medskip\\
J_{\text{in/out},\pm} (v) \le J_{\text{in/out},\pm} (V)
\end{cases}
\end{equation}
holds at $x_c$ where $V = G(U)$.

We next remark that since $G$ is non-decreasing and vanishes at $0$
and $w =v-V$ reaches a zero maximum at $x=x_c$, 
\begin{equation}\label{estim:2}
- I_{\text{in},+} (v-V) \le - \tilde{c}_\alpha \gamma J_{\text{in},+} (v-V)
\end{equation}
holds at $x_c$. 
Indeed, for $\alpha \in (1,2)$ (the proof is the same in the other case),
\begin{align*}
\gamma J_{\text{in},+} (w) (x_c) & = \bar c_\alpha \gamma \int_{B^{\text{in},+}_\gamma} (w (x_c+y) + w(x_c -y) -2 w(x_c)) \frac{dy}{|y|^{d+\alpha}} \\
& = 2 \bar c_\alpha \gamma \int_{B^{\text{in},+}_\gamma} (w(x_c+y)-w(x_c)) \frac{dy}{|y|^{d+\alpha}} \\
& \ge 2 \bar c_\alpha \int_{B^{\text{in},+}_\gamma} (w(x_c+y)-w(x_c)) (y \cdot \hat x_c) \frac{dy}{|y|^{d+\alpha}}.
\end{align*}

Combining \eqref{eq:rewrite}-\eqref{estim:2}, we
get (at $x_c$), 
\begin{align*}
  \partial_t U  \le & |\nabla U| (-
  I_{\inn,+} (V)  - I_{\inn,-} (V) - I_{\out,+} (v) - I_{\out,-} (V) 
+ \tilde{c}_\alpha \gamma J_{\inn,+} (V) ) \\
& + (U - \tilde{c}_\alpha \gamma |\nabla U|) J_{\inn,+} (v) \\  
& + U (J_{\inn,-}(V) + J_{\out,+} (V)+ J_{\out,-} (V)) + \delta \Delta U.
\end{align*}
In view of the choice of $\gamma$, we get
\[ \partial_t U \le -|\nabla U| I (V) + U J (V) + |\nabla U|( -
I_{\out,+} (v) + I_{\out,+} (V)) + \delta \Delta U.\] We now remark
that $-|\nabla U| I(V) = \nabla U \cdot \nabla P$ and $J(V) = \Delta
P$ we get the desired inequality.
\end{proof}

\subsection{Construction of the barrier}
\label{sub:error}

The previous lemma holds true for general barriers $U$. In this
subsection, we specify the barrier we are going to use. We would like
to use $(R(t)-|x|)^2$ but this first try does not work. First the
power $2$ is changed with $\beta$ large enough such that $V=U^{m-1}$
is regular enough. Second, a small $\omega^\beta$ is added in order to 
ensure that the contact does not happen at infinity. Third, a small
slope in time of the form $\omega^\beta t/T$ is added to control some 
error terms. 
\begin{lemma}[Construction of a barrier]\label{lem:barrier}
Assume that
\[ \|u\|_\infty \le 1 \quad \text{ and } \quad 0 \le u_0(x) \le
(R_0-|x|)^\beta_+ \quad \text{ with } \quad R_0 \ge 2 \] for some
$\beta  > \max (2, \alpha (m-1)^{-1})$. There then exist $C>0$
and $T>0$ (only depending on $d,m,\alpha,\beta$) and $U \in
\mathcal{C}^2((0,+\infty) \times (\R^d \setminus B_1))$ defined as
follows,
\begin{equation}\label{def:U}
U (t,x) = \omega^\beta + (R(t)-|x|)_+^\beta + \omega^\beta \frac{t}T 
\end{equation}
where $R(t) = R_0 + Ct$ and $\omega = \omega (\delta)$ small enough,
such that
\begin{enumerate}[i)]
\item the following holds true
\begin{equation}\label{eq:estim-P}
 \nabla P, \Delta P, J_{\inn,+}(V), I_{\out,+}(V),\Delta U \text{ are bounded};
\end{equation}
\item $u$ and $U$ cannot touch at a time
$t<T$ and a point $x_c \in B_1$ or $x_c \notin \bar{B}_{R(t)}$;
\item if $U$ touches $u$ from above at $(t_c,x_c)$ with $t_c < T$ and $x_c \in B_{R(t)}$, then 
\begin{equation}\label{estim:speed}
 C  \lesssim  1 -I_{\out,+}(v) + \frac{\delta}{\omega}.
\end{equation}
\end{enumerate}
\end{lemma}
\begin{proof}
  We first remark that the condition \(R_0 \ge 2\) ensures that the
  contact point is out of $B_1$ since $\|u\|_\infty \le 1$.

  The fact that $U$ is $\mathcal{C}^2$ in $(0,+\infty) \times (\R^d
  \setminus B_1)$ and $V=U^{m-1}$ is $\mathcal{C}^{\alpha+0}$ in $\R^d
  \setminus B_1$ ensures that \eqref{eq:estim-P} holds true. Notice
  that the condition: $\beta (m-1) > \alpha$ is used here.

We should now justify that the contact point cannot be outside
$B_{R(t)}$ at a time $t \in (0,T)$ for some small time $T$
under control. If $|x_c| > R(t)$ and $t_c <T$ then
\begin{equation*}
\begin{cases}
0 \le U \le 2 \omega^\beta \\
\partial_t U = \frac{\omega^\beta}T \\
|\nabla U| = 0 \\
\Delta U = 0.
\end{cases}
\end{equation*}
The contact analysis lemma~\ref{lem:contact} (with $\gamma=1$, say), \eqref{eq:contact}
and \eqref{eq:estim-P} then implies that
\[ \frac{\omega^\beta}T \le |\Delta P | U \lesssim \omega^\beta \] and choosing $T$
small enough (but under control) yields a contradiction.

It remains to study what happens if $t_c <T$ and $x_c \in B_{R(t)} \setminus B_1$. 
In order to do so, we first define $h$ and $H$ as follows: 
\[  U = h^\beta + H^\beta \le 1 \]
with $H^\beta = \omega^\beta t T^{-1} \le \omega^\beta$ for $t \in
(0,T)$. 
Remark that $h \ge \omega \ge H$. In the contact analysis lemma~\ref{lem:contact}, we choose $\gamma$ such that 
\[ \beta c_\alpha \gamma \le \bar{c}_\alpha h. \]
If $x_c \in B_{R(t)} \setminus B_1$,
\begin{equation}\label{eq:der-barrier}
\begin{cases}
\partial_t U = \beta C h^{\beta-1} + \frac{\omega^\beta}T \ge \beta C h^{\beta-1} \\
|\nabla U| = \beta h^{\beta-1} 
\end{cases}
\end{equation}
Combining Lemma~\ref{lem:contact} with
\eqref{eq:estim-P}-\eqref{eq:der-barrier}, we get \eqref{estim:speed}. 
\end{proof}

\subsection{Estimate of the error term}
\label{sub:estim}

\begin{lemma}\label{lem:error}
The following estimate holds true at $x_c$,
\begin{equation}\label{estim:iout}
 -I_{\out,+}(v) \lesssim 
\begin{cases}
 G(2h^\beta) h^{1-\alpha} &\text{ if } \alpha > 1\\
  R_0^{1-\alpha + \eps} & \text{ if } \alpha \le 1
\end{cases}
\end{equation}
for all $\eps >0$.
\end{lemma}
\begin{proof}
We begin with the easy case $\alpha >1$. In this case, we simply write 
\begin{align*}
 -I_{\out,+}(v) & = \int_{\stackrel{|y| \ge \gamma}{y
    \cdot \hat x_c \ge 0}} (v(x_c) - v
(x_c+y)) (y \cdot \hat x_c) \frac{dy}{|y|^{d+\alpha}} \\
& \le v(x_c) \int_{\stackrel{|y| \ge \gamma}{y
    \cdot \hat x_c \ge 0}}  (y \cdot \hat x_c)
\frac{dy}{|y|^{d+\alpha}} \\
& \le v(x_c) \int_{|y| \ge \gamma} \frac{dy}{|y|^{d+\alpha-1}}
\end{align*}
where we used the fact that $v \ge 0$. By remarking
that 
\[v = G(u) = G(h^\beta+H^\beta) \le G (2h^\beta)\] at the contact
point and through an easy and standard computation, we get the desired
estimate in the case $\alpha >1$.

We now turn to the more subtle case $\alpha \in (0,1]$. In this case,
\[ I_{\out,+}(v) = I^* [v] + K \star v \]
 where 
\[ \begin{cases}
-I^* [v] =  \int_{\stackrel{\gamma \le |y| \le 1}{y
    \cdot \hat x_c \ge 0}} - v (x_c+y)) (y \cdot \hat x_c) 
\frac{dy}{|y|^{d+\alpha}} \bigskip\\
 K = \frac{y \cdot \hat x_c}{|y|^{d+\alpha}} \mathbf{1}_{|y| \ge
  1, y \cdot \hat x_c \ge 0}.
\end{cases}\]
We first remark that 
\[ |I^*[v]| \le \|v\|_\infty \int_{B_1} \frac{dy}{|y|^{d+\alpha-1}}
\lesssim 1. \]
We next remark that $K \in L^p (\R^d)$ for all $p > \frac{d}{d-(1-\alpha)}
\ge 1$. Hence, 
\[ | K \star v| \le \| K \|_p \|v\|_q = \| K \|_p \|u\|_{(m-1)q}^{m-1}\]
with $p$ as above and $q^{-1}=1-p^{-1}$.

We next estimate $\|u\|_{(m-1)q}$. Interpolation leads
\[ \|u\|_{(m-1)q}^{m-1} \le \|u\|_1^{\frac1{q}} \|u\|_\infty^{(m-1)-\frac1{q}} 
\le \|u\|_1^{\frac1{q}} \]
since $\|u\|_\infty \le 1$. Finally, we use mass conservation in order to get 
\[ \| u\|_1 = \|u_0 \|_1 \le \int \min (1,(R_0 - |x|)_+^{\beta}) \dd x \le \omega_d R_0^d .\]
Finally, we have
\[ |I_{\out,+}(v)| \lesssim  R_0^{\frac{d}q} \]
for all $q < \frac{d}{1-\alpha}$
which yields the desired result.
\end{proof}
Combining now Lemmas~\ref{lem:barrier} and \ref{lem:error}, we get the following one. 
\begin{lemma}[Estimate of the speed of propagation]\label{lem:reduced}
Assume that 
\[ \|u\|_\infty \le 1 \quad \text{ and } \quad 0 \le u_0(x) \le
(R_0-|x|)^\beta_+ \quad \text{ with } \quad R_0 \ge 2 \] for some
$\beta > \max (2, \alpha(m-1)^{-1} )$.  Then there exists $T>0$ and
$C_0>0$ only depending on dimension, $m$, $\alpha$ and $\beta$ (and
$\eps$ for $\alpha \le 1$) such that, for $t \in (0,T)$, $u$ is
supported in $B_{R_0 + Ct}$ with
\begin{equation}\label{estim:speed-bis}
C = \begin{cases} C_0 & \text{ if } \alpha >1, \\
 C_0 R_0^{1-\alpha -\eps} & \text{ if } \alpha \le 1
\end{cases}
\end{equation}
(for $\eps>0$ arbitrarily small). 
\end{lemma}
\begin{proof}
  In view of Lemma~\ref{lem:barrier}, the parameter $\omega$ is chosen
  so that $\omega \gg \delta$, say $\omega = \sqrt{\delta}$. Now
  Lemmas~\ref{lem:barrier} and \ref{lem:error} imply that if $C$ is
  chosen as indicated in \eqref{estim:speed-bis}, then $u$ remains
  below $U$ at least up to time $T$. Letting $(\omega,\delta)$ go to
  $0$ yields the desired result.
\end{proof}

\section{Proof of Theorem~\ref{thm:main}}
\label{sec:proof}

We can now prove Theorem~\ref{thm:main}. 
\begin{proof}[Proof of Theorem~\ref{thm:main}]
We treat successively the $\alpha >1$ and $\alpha \le 1$. 

\paragraph{First case.}
  In the case $\alpha >1$, if $\|u\|_\infty = \|u_0\|_\infty\le 1$ and 
  \[ u_0(x) \le (R_0 -|x|)_+^\beta, \] then Lemma~\ref{lem:reduced}
  implies that the support of $u$ is contained in $B_{R(t)}$ with
\[ R(t) = R_0 + C_0 t\] 
for some constant $C_0$ only depending on dimension, $m$ and
$\alpha$. Rescaling the solution (see
Lemma~\ref{lem:scaling}), we get 
\[ R(t) = R_0 + C_0 L^{m-1- \frac{\alpha-1}{\beta}} a^{\frac{\alpha-1}\beta} t\]
as soon as 
\[ u_0(x) \le a (R_0-|x|)_+^\beta \quad \text{ and } \quad L= \|u\|_\infty
= \|u_0\|_\infty.\] 

If we simply know that $u_0$ is supported in
$B_{R_0}$ and $\|u\|_\infty = \|u_0\|_\infty = L$, then we can pick
any $a>0$ and $r_1>0$ such that $a r_1^\beta = L$ and get
\[ u_0(x) \le a (r_1 + R_0 -|x|)_+^\beta.\] 
By the previous reasoning, we get that 
\[ R(t) \le R_0 + r_1 + C_0 L^{m-1- \frac{\alpha-1}\beta} a^{\frac{\alpha-1}\beta} 
t = R_0 + r_1 + C_0 L^{m-1} r_1^{1-\alpha} t. \]
Minimizing with respect to $r_1$ yields the desired result in the case
$\alpha >1$. 

\paragraph{Second case.}
We now turn to the case $\alpha \in (0,1]$.  Lemma~\ref{lem:reduced}
yields for $t \in [0,T_1]$ with $T_1 = \frac{R_0}{C_1}$
\[ C_1 \lesssim R_0^{1-\alpha+\eps} \]
(recall that $R_0 \ge 2$). 

We now start with $R_1 = R_0 + C_1 T_1 = 2R_0$ and we get 
\[ C_2 \lesssim (3R_0)^{1-\alpha + \eps} \]
for $t \in [T_1,T_2]$ with 
\[ T_2 - T_1 = \frac{R_0}{C_2}.\]
More generally, for $t \in [T_k,T_{k+1}]$, 
\[ C_k \simeq ((k+1)R_0)^{1-\alpha +\eps} \simeq (kR_0)^{1-\alpha+\eps}\]
with 
\[ T_{k+1}-T_k = \frac{R_0}{C_k} \simeq \frac{R_0^{\alpha -
    \eps}}{(k+1)^{1-\alpha + \eps}} .\]
We readily see that the series $\sum_k (T_{k+1}-T_k)$
diverges. More precisely,
\[ T_k \simeq (kR_0)^{\alpha -\eps}.\]
Moreover, we get that the function $u$ is supported in
$B_{R(t)}$ with  
\[ R(t) -R_0 \lesssim  kR_0 + C_k (t-T_k) \lesssim 
(T_k)^{\frac1{\alpha-\eps}} + (T_k)^{\frac{1-\alpha+\eps}{\alpha-\eps}}
    t \lesssim  t^{\frac{1}{\alpha-\eps}}\]
for $t \in [T_k,T_{k+1}]$. Hence, we get the result but not with the
right power. Precisely, for $L=1$ and 
\[ 0 \le u_0(x) \le (R_0 -|x|)_+^\beta \]
we get
\[ R(t) = R_0 + C_0 t^\beta \]
with $\beta > \frac1\alpha$. Rescaling and playing again with $r_1$
and $a$ such that $a r_1^\beta =L$ yields the desired result in the case
$\alpha <1$.  The proof of
the theorem is now complete.
\end{proof}

\paragraph{Acknowledgements.} The author wishes to thank P.~Biler and G.~Karch for fruitful discussions during the preparation of this note. 
He also thanks a referee for a very attentive reading of the proofs which leads to an improved version of the note. 

\bibliographystyle{plain} \bibliography{fsp}

\def\cprime{$'$}
\begin{thebibliography}{1}

\bibitem{bik0}
Piotr Biler, Cyril Imbert, and Grzegorz Karch.
\newblock {Barenblatt profiles for a nonlocal porous medium equation.}
\newblock {\em C. R., Math., Acad. Sci. Paris}, 349(11-12):641--645, 2011.

\bibitem{bik}
Piotr Biler, Cyril Imbert, and Grzegorz Karch.
\newblock Nonlocal porous medium equation: Barenblatt profiles and other weak
  solutions.
\newblock HAL hal-795420, 2013.

\bibitem{cv10a}
Luis Caffarelli and Juan~Luis V\'azquez.
\newblock Nonlinear porous medium flow with fractional potential pressure.
\newblock {\em Arch. Ration. Mech. Anal.}, 202(2):537--565, 2011.

\bibitem{stv}
Diana Stan, F{\'e}lix del Teso, and Juan~Luis V{\'a}zquez.
\newblock Finite and infinite speed of propagation for porous medium equations
  with fractional pressure.
\newblock {\em C. R. Math. Acad. Sci. Paris}, 352(2):123--128, 2014.

\end{thebibliography}

\end{document}